\documentclass[a4paper]{article}

\usepackage[pages=all, color=black, position={current page.south}, placement=bottom, scale=1, opacity=1, vshift=5mm]{background}

\usepackage[margin=1in]{geometry} 

\usepackage{amsmath}
\usepackage{amsthm}
\usepackage{amssymb}
\usepackage{ bbold }
\usepackage{xcolor}
\usepackage{multirow}
\usepackage{booktabs}
\makeatletter
\providecommand{\tbl}[2]{%
  \begin{table}[ht]
    \centering
    \caption{#1}
    #2
  \end{table}%
}
\makeatother

\input{fitch.sty}

\newcommand{\ind}{\perp\!\!\!\!\perp}

\usepackage[utf8]{inputenc}
\usepackage{hyperref}
\hypersetup{
	unicode,
	pdfauthor={Reid Dale},
	pdftitle={Budget Your Alpha Inference Policies for Data Registries},
	pdfsubject={A simple article template},
	pdfkeywords={article, template, simple},
	pdfproducer={LaTeX},
	pdfcreator={pdflatex}
}

\theoremstyle{plain}
\newtheorem{theorem}{Theorem}
\newtheorem{corollary}[theorem]{Corollary}

\newtheorem{proposition}[theorem]{Proposition}

\theoremstyle{definition}
\newtheorem{definition}[theorem]{Definition}
\newtheorem{example}[theorem]{Example}

\usepackage{tikz}
\usetikzlibrary{matrix}

\usepackage{graphicx, color}
\graphicspath{{fig/}}

\usepackage{algorithm, algpseudocode} 
\usepackage{mathrsfs} 

\usepackage{lipsum}

\title{Data Gluttony: Epistemic Risks, Dependent Testing and Data Reuse in Large Datasets}
\author{Reid Dale$^1$, Jordan Rodu$^2$, Maria E. Currie$^3$, Mike Baiocchi$^4$}

\date{
	$^1$Stanford University School of Medicine Department of Cardiothoracic Surgery \\ \texttt{reiddale@stanford.edu}\\%
    $^2$University of Virginia Department of Statistics\\
    \texttt{jsr6q@virginia.edu}\\
    $^3$Stanford University School of Medicine Department of Cardiothoracic Surgery \\
    \texttt{mecurrie@stanford.edu}\\
     $^4$Stanford University Department of Epidemiology $\&$ Population Health \\ \texttt{baiocchi@stanford.edu}\\
}

\begin{document}
\maketitle

\begin{abstract}
Large-scale registries have collected vast amounts of data which has enabled investigators to efficiently conduct studies of observational data. However, we demonstrate how data reuse leads to positive dependence among the inferential tasks and cascading inferential errors. 

    Common practice is for investigators to use all data meeting the inclusion criteria of their study to perform their analysis. We term this common practice data gluttony. It has apparent formal justification insofar as this approach maximizes per-study power. But this comes at a cost: data reuse affects the \textit{shape} of the distribution of inferential errors. Using the theory of risk orderings we demonstrate how positively dependent testing procedures result in strictly riskier distributions of inferential error. 
    
    We identify two remedies to this state of affairs: research portfolio optimization and what we term data temperance. Research portfolio optimization requires that we formulate the enterprise of inference in a utility theoretic framework: associated to each hypothesis to be evaluated is some utility dependent on its truth as well as the impact of the statistical decision rendered on the basis of the data. Under certain models of data governance, this approach can be used to optimally allocate data usage across multiple inferential tasks. 

    On the other hand, data temperance is a more flexible strategy for managing the distribution of inferential errors. Data temperance is the principle that an investigator use only as much data as is necessary to perform the task at hand. This is possible due to the diminishing marginal returns in power and precision in sample size. In contrast to portfolio optimization, data temperance can be effectively applied in a federated manner. Moreover, we analyze the effectiveness of data temperance at reducing the dependence across testing and develop a theory of the capacity of a static database to sustain large numbers of inferential tasks with low probability of inducing pairwise dependent testing procedures.

    Registries are an invaluable resource. We must ensure that they \textit{remain} valuable by mitigating the risk of cascading inferential errors.
    
\end{abstract}

\newpage 

	\tableofcontents

\newpage 
\section{Introduction}

The use, and particularly the \textit{reuse}, of large datasets for inferential tasks in medical, economic, and machine learning contexts has enabled investigators to conduct many studies. The benefits of such datasets are promising: (i) increased sample size allows for improved power, precision, and training data; (ii) the bottleneck of data collection is obviated for amenable tasks; (iii) such datasets tend to be collected from federated reporting sites, representing a broader swath of the intended superpopulation; (iv) facilitates the analysis of rare subgroups that would be otherwise difficult to study; (v) large observational data is required to probe hard-to-randomize treatments such as smoking health outcomes, impact of state-level health care changes, or donor-recipient matching in transplantation.

But are these databases and registries the panacea that was promised? The central contention of this paper is that the \textit{reuse} of data and common practices in hypothesis development pose the systemic risk of statistical analyses that exhibit nontrivial dependence structure, leading to long-tailed distributions of error amenable to portfolio-theoretic analysis.

These benefits of these datasets are tremendous, but the reuse of these datasets poses novel risks that require consideration. First, unrestricted data access allows for untraceable $p$-hacking: it is trivial to write a script to perform statistical analyses on a dataset until a decision threshold is obtained for some hypothesis which may avail post hoc justification in a publication. Recent work has identified a precipitous rise in the number of formulaic research articles generated from multiple publicly available datasets, raising concerns about the rigor and reproducibility of research generated in this manner (\cite{suchak_explosion_2025}, \cite{spick_quantifying_2025}).

In this paper, we articulate a distinct concern of underlying multiplicity concomitant to the rapid increase in the number of research projects reusing datasets. The ease of use of such datasets enables the following research cycle: Researcher $A$ may evaluate hypothesis $H_0$ using data $D$ and publishes their result. Researcher $B$ sees this and constructs a related hypothesis $H'_0$, which they can conveniently evaluate on the same dataset $D$. Thus, the generation of $H'_0$ is dependent on a decision procedure made using the same database $D$, resulting in an indirect form of data peeking. This research cycle is distinct from the traditional framework which gave rise to existing error protection procedures (e.g., p-values, multiple hypothesis testing corrections). Traditionally, novel data sets were generated for each study (e.g., a randomized controlled trial, a survey using a sampling procedure) and then used to evaluate a specified number of hypotheses. In the traditional framework, subsequent studies would then rely on novel data sets.

The status quo for utilizing these datasets can be articulated as follows: to evaluate null hypothesis $H_0$ using data from $D$, analyze all elements of $D$ meeting the inclusion criteria for the study. We term this policy \textit{data gluttony} as it seeks to maximize the data used in the evaluation of $H_0$. By contrast, we articulate a position termed \textit{data temperance}, which can be described as using only a fraction of the data that is needed to achieve the objectives of the study.

There is a long and rich tradition of analyzing multiplicity in statistical procedures \textit{within} a particular study (e.g. multiple corrections) or in the context of adaptive evaluations of a hypothesis (e.g. sequential testing). These approaches are often couched in terms of managing expected rates of error events (e.g. FWER, FDR, etc.) such as  \cite{benjamin_redefine_2018}, \cite{de_ruiter_redefine_2019},  \cite{thompson_dataset_2020}, \cite{foster_investing_2008}, and \cite{johari_always_2019}). By contrast, we approach this phenomenon from the perspective of managing a \textit{portfolio} of analyses --- i.e., \textit{across} studies --- leveraging utility-theoretic machinery  (as in \cite{lakens_justify_2018}) to identify techniques to mitigate the epistemic risks engendered by the reuse of data.


\section{Data Gluttony}

\subsection{Data Gluttony Defined and Defended}

We begin our investigation by articulating the default policy of maximal data usage---data gluttony---and a formal justification for it.

When conducting a study using observational or retrospective data, investigators aim to make the most of the data available to them. Typically, this manifests in using as much data as is available meeting the inclusion/exclusion criteria for the study. This policy has formal justification in the context of the evaluation of a family of hypotheses.

Suppose that we want to simultaneously achieve:
\begin{enumerate}
\item Type I error control at level $\alpha$ per study, and
\item Maximum average and per-study power.
\end{enumerate}

We assume that for each $H_{0i}$ there is a family of tests depending on a vector of integers encoding sample size $\pi(\theta, \overline{n})$ such that if $\overline{n} \leq \overline{m}$ in the pointwise partial order then $\pi(\theta, \overline{n}) \leq \pi(\theta, \overline{m})$.

This holds, for example, for the power function for comparing the mean of two treatments using a $t$-test. When this formal condition holds, the argument for Theorem~\ref{thm:glutton-defense} holds.

\begin{theorem}\label{thm:glutton-defense}
Let $H_{01},\ldots, H_{0k}$ be a set of null hypotheses on a parameter space $\Theta$ and let $D$ be a shared dataset. Let $T_1(D_1),\ldots, T_k(D_k)$ be a family of tests of level $\alpha$ evaluating $H_{0i}$ on some data $D_i \subseteq D$ based on its inclusion criteria.

For each $H_{0i}$ let $D_{i}^{\max}$ be the maximal subset of $D_i$ meeting the inclusion criteria for the study. Then the expected rate of Type II errors is minimized when $D_i = D_i^{\max}$.
\end{theorem}

\begin{proof}
Since the power of each test is maximized by using the largest subset of sampled data meeting the inclusion criteria for the test, we can minimize the rate of Type II errors among the family of tests by myopically maximizing the data usage of each test.

More formally, for each $\theta \in \Theta$ and $H_{0i}$, the power of the test $T_i(X)$, $\pi_i(\theta;X)$, satisfies
\begin{equation}
X\subseteq Y \Rightarrow \pi_i(\theta;X) \leq \pi_i(\theta;Y)
\end{equation}
for all $\theta$ and $X,Y$ drawn from the superpopulation. Consequently, uniformly in $\theta$, picking the largest subset of $D$ meeting the inclusion criteria for $H_{0i}$ maximizes $\pi_i(\theta,X)$.

Now, for each $\theta \in \Theta$ and $D_i\subseteq D$, we have that 
\begin{equation}
\mathbb{E}\left[\frac{\sum (1-\pi_i(\theta,D_i))}{n}\right] = 1- \frac{\sum \pi_i(\theta,D_i)}{n}.
\end{equation}
Since $\pi_i(\theta,D_i)$ does not depend on any other test, a sufficient criterion for minimizing the error rate is maximizing $\pi_i(\theta,D_i)$.
\end{proof}

The arguments offered in favour of data gluttony rest on a shared flavour of assumption: that our disposition toward error be governed by a specific utility function. For the hypothesis testing glutton, the utility function of a study $S$ is 
\begin{equation}
u(S) = -(\#E_{II}(S) +\#E_{I}(S))
\end{equation}
where $\#E_I$ and $\# E_{II}$ are the counts of type I and type II errors respectively. Then the constrained optimization of maximizing $\mathbb{E}[u]$ subject to the constraint that each test is performed at level $\alpha$ (independently of the sample size) is achieved by the data glutton.

Perhaps data gluttony is the appropriate attitude to take. The following section is intended to provide reasons to question data gluttony. As we will see, the dependence structure of a family of testing procedures affects the expected utility of a portfolio of studies for a wide class of utility functions.

\subsection{Portfolio-Theoretic Perspectives on Data Gluttony}

In this section we focus our attention on the per-study linear utility function
\begin{equation}
u(S) = -(\#E_{II}(S) +\#E_{I}(S))
\end{equation}
and examine its relation to dependence among testing procedures in a portfolio of studies. For a given set $\mathcal{P}$ of studies $S$, the linear utility of the portfolio $\mathcal{P}$ is defined to be
\begin{equation}
u(\mathcal{P}) = \sum_{S\in \mathcal{P}} u(S).
\end{equation}

First, let $\Theta$ be a parameter space and $\{H_{0i}\}_{i\leq n}\subseteq \Theta$ be a set of distinct null hypotheses, and let $H_{1i} = \Theta\setminus H_{0i}$ be the alternative hypotheses. We assume that each study $S_i$ is evaluating a single null hypothesis $H_{0i}$, using some test and data.

Observe that for a given study $S$, the error events $E_{I}(S)$ and $E_{II}(S)$ are disjointly supported: for $E_{I}(S)$ to occur, $H_0(S)$ must be true and for $E_{II}(S)$ to occur $H_1(S)$ must be true. Consequently, $u(S)$ is binomially distributed with probability $\mathrm{pr}(u(S) = 1) = p(\theta)$ depending on the true value of $\theta$.

This decomposition allows us to explore two useful special cases before turning to the general case: (i) the distribution of errors under the global null, (ii) the distribution of errors under the global alternative, and (iii) an expected utility theoretic approach to data usage.

\subsection{Case Study I: Type I Errors under the Global Null}\label{sec:typeI_error}

Let $\{H_{0i}\}_{i\leq n}\subseteq \Theta$ be a set of null hypotheses, and assume that for each study $S_i$ evaluating $H_{0i}$ the test has level $\alpha$ for all sufficiently large $n$. Assume further that each $H_{0i}: \theta_i = \theta_{0i}$ is a two-sided hypothesis, so that $\mathrm{pr}(E_{I}) = \alpha$.

For the remainder of the section, we assume the global null $\bigcap H_{0i}$ is true and investigate the relationship between the dependence of testing procedures and properties of the linear utility function.

Assuming the global null, for each study the utility function $u$ reduces to
\begin{equation}
u(S) = -\#E_{I}(S).
\end{equation}

For each study $S$, $u(S)$ is distributed by
\begin{equation}
u(S) =
\begin{cases}
0 & \text{with probability } 1-\alpha \\
-1 & \text{with probability } \alpha
\end{cases}
\end{equation}
and consequently the portfolio-level utility is given by 
\begin{equation}
u(\mathcal{P}) = \sum_{S\in \mathcal{P}} u(S)
\end{equation}
is the negative sum of Bernoulli random variables.

By linearity of expectation, 
\begin{equation}
\mathbb{E}[u(\mathcal{P})] = \sum_{S\in \mathcal{P}} \mathbb{E}[u(S)] = -\alpha|\mathcal{P}|
\end{equation}
and is invariant under the dependence between the studies comprising $\mathcal{P}$.

By being invariant under the dependence of the underlying error events, the utility function $u$ adopts a specific risk attitude toward the incidence of Type I errors: it is risk-neutral in the expected number of Type I errors.

We argue that, under the global null, the indifference to portfolios of equal expected number of Type I errors can lead to riskier-than-necessary dependence in the findings.

\begin{enumerate}
\item The utility function $u$ is indifferent to any portfolios $\mathcal{P}$ with identical expected Type I error rates.
\item We can construct portfolios $\mathcal{P}_1$ and $\mathcal{P}_2$ evaluating the same hypotheses using data drawn from $D$ such that (a) $\mathcal{P}_1$ and $\mathcal{P}_2$ have equal expected counts of Type I errors, (b) the studies comprising $\mathcal{P}_1$ have independent errors while $\mathcal{P}_2$ has positively dependent errors, and (c) $\mathcal{P}_1$ would be preferred by all risk-averse agents to $\mathcal{P}_2$.
\item Consequently, the preference relation between portfolios exhibited by $u$ fails to be risk-averse, and is inefficient insofar as both risk-neutral and risk-averse agents would be satisfied by selecting $\mathcal{P}_1$ over $\mathcal{P}_2$.
\end{enumerate}

The argument hinges crucially on the data usage strategy. While using disjoint sets of data drawn from the same distribution guarantees the independence of the test procedures, the reuse of data can result in positively dependent testing procedures that make the portfolio riskier in a precise mathematical sense given by stop-loss theory (Section \ref{subsec:stoploss}).

This argument alone does not settle the debate about whether the linear utility function is satisfactory. For one, we are assuming the global null hypothesis with certainty. If the null hypothesis were known with certainty, there would be no point in performing the test! Moreover, by concentrating on the global null, we have precluded the possibility of Type II errors from occurring. While (approximately) level $\alpha$ tests can be conducted at all sufficiently large sample sizes, for a fixed level $\alpha$ there is a tradeoff between sample size and power that must be accounted for. This tradeoff is discussed and explored in Section \ref{sec:power_tradeoff}.

\subsection{Dependence of Type I Errors, Stop Loss Premiums, and Risk Aversion}\label{subsec:stoploss}

We now explore the consequences of this choice in the context of two studies $S_1$ and $S_2$ comprising the portfolio $\mathcal{P}$. Let $E_1 = \#E_I(S_1)$ and $E_2 = \#E_I(S_2)$.

Given $E_1$ and $E_2$, the probability distribution for the count of errors is expressible as a function of $\mathrm{pr}(E_1)$, $\mathrm{pr}(E_2)$, and the conditional probability $\mathrm{pr}(E_2|E_1)$ (we refer the interested reader to Theorem \ref{thm:dependent_error_probs} for a proof).

\tbl{Distribution of Errors for Two Tests (Conditional Probability Rewrite)}{
\begin{tabular}{cc}
\textbf{Count} & \textbf{Probability} \\
0 & $1-\mathrm{pr}(E_2) - \mathrm{pr}(E_1)(1-\mathrm{pr}(E_2|E_1))$ \\
1 & $\mathrm{pr}(E_2) + \mathrm{pr}(E_1)\times(1-2\mathrm{pr}(E_2|E_1))$\\
2 & $\mathrm{pr}(E_1)\times \mathrm{pr}(E_2|E_1)$ \\
\end{tabular}}
\label{tab:dependent_error_distribution_condprob}

We can illustrate the effect of nontrivial dependence between $E_1$ and $E_2$ on the distribution of errors in Figure \ref{fig:distribution_of_errors}. We assume that $\mathrm{pr}(E_1) = \mathrm{pr}(E_2) = \alpha$ but that $\mathrm{pr}(E_2|E_1)$ ranges from $0$ to $1$. Observe that while the expected number of errors is constant, the higher the conditional probability $\mathrm{pr}(E_2|E_1)$ the more likely there will be $2$ errors in the portfolio.

\begin{figure}[htbp]
    \centering
    \includegraphics[width=0.8\textwidth]{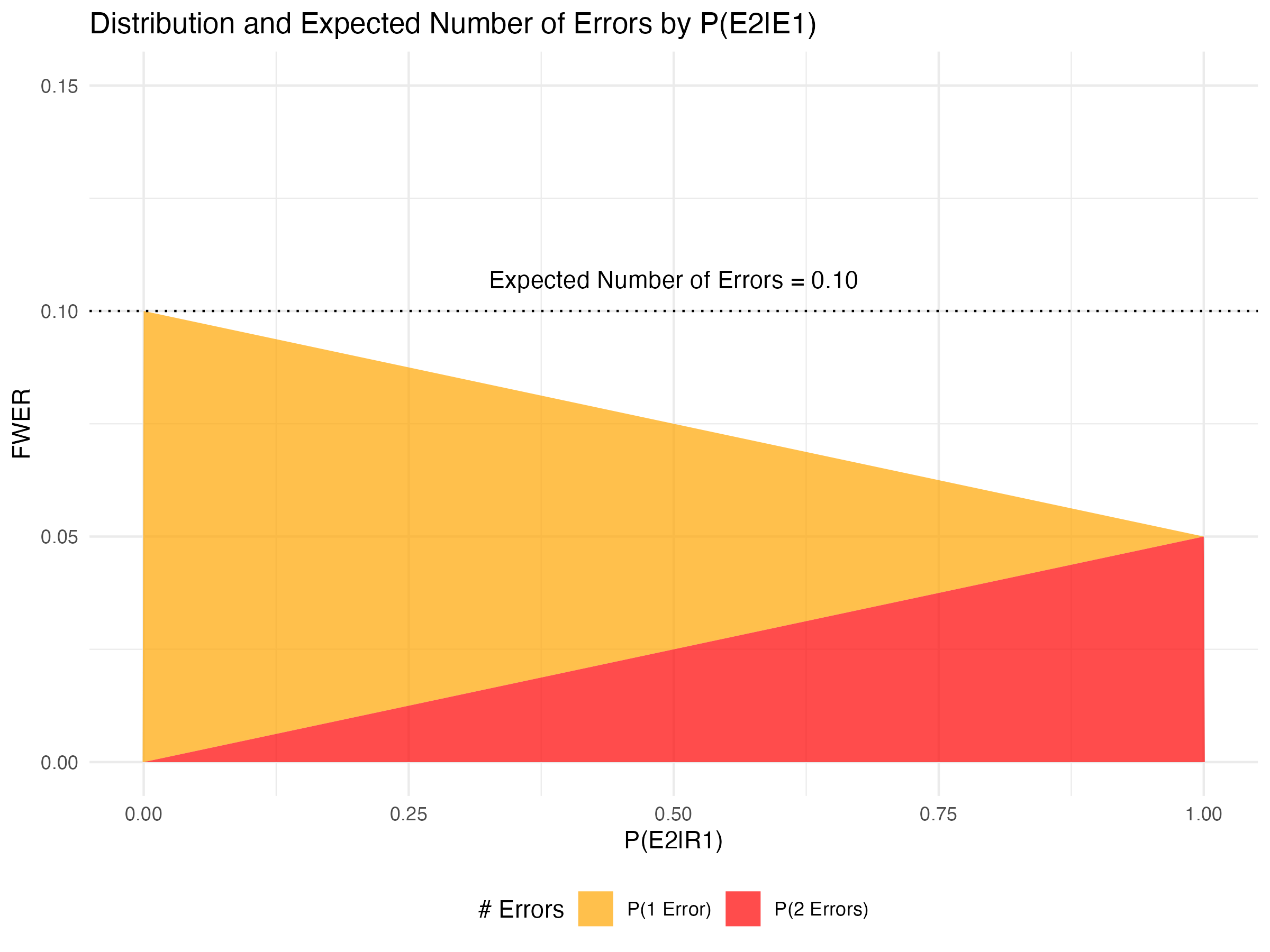}
    \caption{Distribution of Errors Under Different Dependence Structures}
    \label{fig:distribution_of_errors}
    \end{figure}

The obvious question becomes: is there a way to evaluate the impact of this dependence on the preferences of portfolios among risk-averse agents? While we will discuss the management of portfolios under explicit utility functions via expected utility theory, stop-loss theory provides sufficient conditions for risk-averse agents to unanimously prefer one portfolio over another.

The basic intuition of stop-loss theory stems from the theory of reinsurance: if portfolios $\mathcal{P}_1$ and $\mathcal{P}_2$ have equal means, then if for each possible loss $L$ cost to reinsure $\mathcal{P}_1$ to cover a loss of $L$ is at most as expensive as reinsuring $\mathcal{P}_2$ then the risk-averse agent should (weakly) prefer $\mathcal{P}_1$ to $\mathcal{P}_2$.

In this section we discuss how the theory of risk orderings relates to preferences regarding the accumulation of inferential errors. In the following, we assume that the agent's utility depends on only the count of errors; the more general decision-theoretic framework under a fixed utility function $u$ will be discussed in a subsequent section.

Following \cite{kaas_modern_2008}, associated to a random variable $X$ we can associate to it the stop-loss premium function $\rho_X$ defined by the formula
\begin{equation}
\rho_X(L) = \mathbb{E}[(X-L)_{+}]
\end{equation}

When $X$ is a discrete random variable, this reduces to
\begin{equation}
\rho_X(L) = \sum_{x>L}(x-L)\mathrm{pr}(X = x).
\end{equation}

In actuarial terms, the function $\rho_X(L)$ represents the premium needed to cover a loss incurred of size $> L$. The stop-loss premium function also turns out to have an important decision-theoretic interpretation: comparisons of the loss premium function between two random variables determines whether every risk-averse agent prefers $X$ to $Y$.

\begin{theorem}[Kaas et al. 2008, Theorem 7.3.10]
Let $X$ and $Y$ be random variables. Then 
\begin{equation}
\forall L \; \rho_X(L) \leq \rho_Y(L) 
\end{equation}
if and only if for every concave increasing utility function $u$
\begin{equation}
\mathbb{E}[u(-X)]\geq \mathbb{E}[u(-Y)].
\end{equation}
\end{theorem}

Consequently, the stop-loss order between two random variables is defined to be:

\begin{definition}
We say that $X$ is stop-loss smaller than $Y$, written $X\leq_{SL} Y$, just in case
\begin{equation}
\forall L\; \rho_X(L) \leq \rho_Y(L).
\end{equation}
We say that a portfolio $\mathcal{P}$ is stop-loss smaller than portfolio $\mathcal{Q}$ just in case the error distributions satisfy $\mathcal{E}_\mathcal{P} \leq_{SL}\mathcal{E}_\mathcal{Q}$.
\end{definition}

Concretely, if a portfolio $\mathcal{P}$ is stop-loss smaller than portfolio $\mathcal{Q}$, every risk-averse agent prefers $\mathcal{P}$ to $\mathcal{Q}$.

Turning now to the example of two (possibly dependent) error events $E_1$ and $E_2$, we investigate how varying $\mathrm{pr}(E_2|E_1)$ affects the stop-loss order. We first calculate the stop-loss premium functions associated with the portfolio as a function of $\mathrm{pr}(E_1)$, $\mathrm{pr}(E_2)$, and $\mathrm{pr}(E_2|E_1)$ (we refer the interested reader to Theorem \ref{thm:dependent_stop_loss} for a proof).

\tbl{Stop Loss Premiums for the Two-Asset Error Portfolio}{
\begin{tabular}{cc}
\textbf{Count} & $\rho_{E_\mathcal{P}}(L)$ \\
0 & $\mathrm{pr}(E_1) + \mathrm{pr}(E_2)$\\
1 & $\mathrm{pr}(E_1)\times \mathrm{pr}(E_2|E_1)$\\
2 & $0$\\
\end{tabular}}
\label{tab:dependent_error_stoploss}

\begin{example}
Let $E_1$ and $E_2$ be error events with probability
\begin{equation}
\mathrm{pr}(E_1) = \mathrm{pr}(E_2) = \alpha.
\end{equation}
Thus, we can identify each $x\in [0,1]$ with a portfolio $\mathcal{P}_{x}$ such that the dependence between $E_1$ and $E_2$ in $\mathcal{P}_{x}$ is $\mathrm{pr}(E_2|E_1) = x$.

Then the stop-loss premium function reduces to Table \ref{tab:dependent_error_stoploss_alpha}.

\tbl{Stop Loss Premiums for the Two-Asset Error Portfolio}{
\begin{tabular}{cc}
\textbf{Count} & $\rho_{E_\mathcal{P}}$ \\
0 & $2\alpha$\\
1 & $\alpha\mathrm{pr}(E_2|E_1)$\\
2 & 0 \\
\end{tabular}}
\label{tab:dependent_error_stoploss_alpha}

The stop-loss premium functions are plotted for $\mathrm{pr}(E_2|E_1)$ ranging from $0$ to $1$ in Figure \ref{fig:stoploss_premiums}.

\begin{figure}[htbp]
    \centering
    \includegraphics[width=0.8\textwidth]{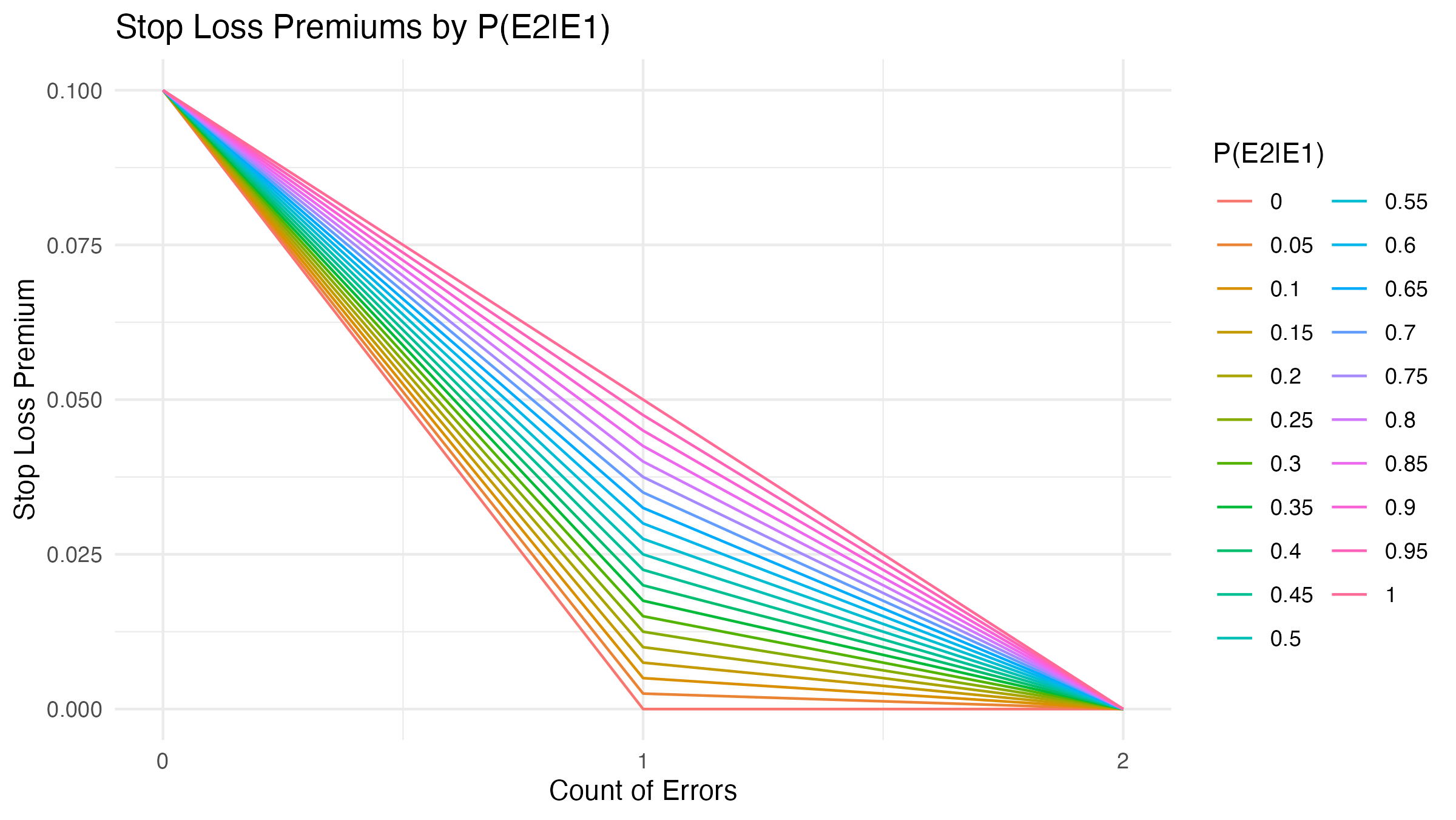}
    \caption{Stop Loss Premiums as a Function of $\mathbb{P}(E_2|E_1)$}
    \label{fig:stoploss_premiums}
\end{figure}

In this setting, the stop-loss premium functions are linearly ordered. First, for each value of $\mathrm{pr}(E_2|E_1)$, $\rho_{E_{\mathcal{P}}}(0) = 2\alpha$ and $\rho_{E_{\mathcal{P}}}(2) = 0$, so there is no difference in stop-loss premiums at those values. However, the stop-loss premium at $1$ differs, and $\mathcal{P}_x$ is stop-loss smaller than $\mathcal{P}_y$ if and only if $x \leq y$. Consequently, the stop-loss premium function is minimized within a set of feasible portfolios exactly when $\mathrm{pr}(E_2|E_1)$ is minimized.
\end{example}

As we will see in the next section, standard measures of Type I error management are either risk neutral or risk-seeking.

\subsection{How to Get Dependent Testing Procedures and Errors through Data Reuse}

We begin by articulating how common practices in the research cycle naturally lead to positively dependent inferential testing procedures and error accumulation.

The previous section describes how the presence of nontrivial dependence between the error events $R_i$ lead to increased probability of a greater count of errors. In this section we show that nontrivial dependence structure is a common scenario facing the portfolio of studies generated by a single dataset $D$. Thus, we isolate and exemplify common testing procedures that yield nontrivial dependence structures through data reuse.

First, we observe how drawing disjoint samples yields independent Type I and Type II error events.

\begin{proposition}
Let $X =(X_1,\ldots, X_n,\ldots)$ be independent. Let $D_1, D_2\subset \mathbb{N}$ be disjoint finite sets. Let $t$ and $s$ be real-valued test statistics on $D_1$ and $D_2$ respectively.

Then every pair of events $E_1,E_2$ of the form 
\begin{equation}
\begin{array}{c}
E_1:t(D_1) \;\square_1\; c_1 \\
E_2:s(D_2)\;\square_2\; c_2,
\end{array}
\end{equation} 
where $\square_1,\square_2 \in \{>,\geq,=,\leq,<\}$ is any pair of relations, are independent.
\end{proposition}

\begin{proof}
Since $D_1$ and $D_2$ are disjoint, by the iid assumption the joint distributions of $D_1$ and $D_2$ are jointly independent. Thus, for every measurable $t:(X_i)_{i\in D_1}\to\mathbb{R}$ and $s:(X_j)_{j\in D_2}\to\mathbb{R}$ we have that the random variables $t$ and $s$ are independent. Consequently, the events $E_1:|t(D_1)| \;\square_1\; c_1$ and $|s(D_2)|\;\square_2\; c_2,$ are independent.
\end{proof}

Conversely, reuse of data can result in positively dependent test statistics. We consider the simple case of the reuse of a control group in a two-sample $t$-test for the difference in group means.

\begin{proposition}
Let $C$ be a control group and $T_1,T_2$ be treatment groups. Assume that 
\begin{equation}
\begin{array}{cc}
C: y_i \sim Y & \text{iid with $2n$ units,} \\
T_1: x_i \sim X & \text{iid with $n$ units,} \\
T_2: z_i \sim Z & \text{iid with $n$ units,} 
\end{array}
\end{equation}
Let $A,B\subseteq [2n]$ be fixed subsets of size $n$. Let $k = k(A,B) = |A\cap B|.$

Suppose further that the standard deviations of $X,Y,$ and $Z$ are the same and known:
\begin{equation}
\sigma := \sigma_X = \sigma_Y = \sigma_Z > 0.
\end{equation}

We consider the design where we compare $T_1$ vs. $C$ using $x = (x_i)_{i\leq n}$ and $y_A = (y_j)_{j\in A}$, and $T_2$ vs. $C$ using $z = (z_i)_{i\leq n}$ and $y_B = (y_j)_{j\in B}$.

Let 
\begin{equation}
\begin{array}{c}
t_1 = \frac{\overline{x}-\overline{y_A}}{\sqrt{2\sigma^2 n^{-1}}} \\
t_2 = \frac{\overline{z}-\overline{y_B}}{\sqrt{2\sigma^2 n^{-1}}}
\end{array}
\end{equation}
be the two-sample $z$-statistics for the above design. Then
\begin{equation}
\mathrm{cov}(t_1,t_2) = \frac{k}{2n} \geq 0.
\end{equation}

The Pearson correlation $\rho(t_1,t_2)$ satisfies
\begin{equation}
\rho(t_1,t_2) = \frac{k}{2n}.
\end{equation}
\end{proposition}

\begin{proof}
\textit{Covariance Bounds}: We simplify $\mathrm{cov}(t_1,t_2)$ as follows. First, by construction observe that
\begin{equation}
\mathrm{cov}(t_1,t_2) = \frac{n}{2\sigma^2} \mathrm{cov}(\overline{x}-\overline{y_A},\overline{z}-\overline{y_B}).
\end{equation}

Then, 
\begin{equation}
\begin{array}{ccc}
\mathrm{cov}(\overline{x}-\overline{y_A},\overline{z}-\overline{y_B})& = & \mathbb{E}[(\overline{x}-\overline{y_A})(\overline{z}-\overline{y_B}))] - \mathbb{E}[\overline{x}-\overline{y_A}]\mathbb{E}[\overline{z}-\overline{y_B}] \\
& = & (\mathbb{E}[\overline{x}\overline{z}] - \mathbb{E}[\overline{x}\overline{y_B}] - \mathbb{E}[\overline{z}\overline{y_A}] + \mathbb{E}[\overline{y_A}\overline{y_B}] ) \\
& & -(\mathbb{E}[\overline{x}]\mathbb{E}[\overline{z}]-\mathbb{E}[\overline{x}]\mathbb{E}[\overline{y_B}] - \mathbb{E}[\overline{z}]\mathbb{E}[\overline{y_A}] + \mathbb{E}[\overline{y_A}]\mathbb{E}[\overline{y_B}]).
\end{array}
\end{equation}

Since $x\ind z$, $x\ind y_B$, $z\ind y_A$ this reduces to 
\begin{equation}
\mathrm{cov}(\overline{x}-\overline{y_A},\overline{z}-\overline{y_B}) = \mathrm{cov}(\overline{y_A},\overline{y_B}).
\end{equation}

We now further decompose $\overline{y_A}$ and $\overline{y_B}$ into the sums of independent random variables. Let
\begin{equation}
\begin{array}{c}
y_{\cap} = \frac{1}{n}\sum_{i\in A\cap B} y_i \\
y_{A\setminus B} = \frac{1}{n}\sum_{i\in A\setminus B} y_i \\
y_{B\setminus A} = \frac{1}{n}\sum_{i\in B\setminus A} y_i.
\end{array}
\end{equation}

Then $\overline{y_A} = y_\cap + y_{A\setminus B}$ and $\overline{y_B} = y_\cap + y_{B\setminus A}$, with $y_\cap, y_{A\setminus B},$ and $y_{B\setminus A}$ all pairwise independent.

By the same argument as above, 
\begin{equation}
\begin{array}{ccc}
\mathrm{cov}(\overline{y_A},\overline{y_B}) & = & \mathrm{cov}(y_{\cap},y_{\cap}) \\
& = & \mathrm{var}(y_{\cap}).
\end{array}
\end{equation}

Now, $\mathrm{var}(y_{\cap})$ is the variance of a sum of $k$ iid random variables with distribution $\frac{1}{n} Y$ and therefore
\begin{equation}
\mathrm{var}(y_{\cap}) = \frac{k \sigma_Y^2}{n^2}.
\end{equation}

Hence
\begin{equation}
\mathrm{cov}(t_1,t_2) = \frac{n}{2\sigma^2} \frac{k\sigma^2}{n^2} = \frac{k}{2n}.
\end{equation}

\textit{Correlation Bounds}: Having calculated $\mathrm{cov}(t_1,t_2)$, to bound $\rho(t_1,t_2)$ it suffices to observe that since $t_1, t_2\sim \mathcal{N}(0,1)$, $\sigma_{t_1}=\sigma_{t_2} =1$ so that 
\begin{equation}
\rho(t_1,t_2) = \mathrm{cov}(t_1,t_2) = \frac{k}{2n}.
\end{equation}
\end{proof}

Consequently, data reuse results in positively dependent testing procedures. We turn now to examples of how data reuse can affect the incidence of errors in multiple applied contexts.

Our first example is a discussion of an example from \cite{fay_statistical_2022}). In this example, we compare the distribution of errors between two study portfolios comparing novel treatments to a control group. In one design, the control group is reused across all treatments thereby inducing dependent testing procedures through data reuse. In the other, disjoint sets of control units are used.

\begin{example}
Suppose that we have a control $C$ and several treatment groups $T_1,\ldots, T_m$, with $7$ treatments. We assume that these groups share the same distribution $\sim \mathcal{N}(0,1)$.

In Design 1, we test for the difference in mean between $T_i$ and $C$, where both $T_i$ and $C$ have 100 patients and each $T_i$ is compared to the same $C$ group using a $t$-test at the $\alpha = 0.05$ level.

In the second design, for each $T_i$ is compared to a control group $C_i$ disjoint from the rest of the treatments using a $t$-test at the $\alpha = 0.05$ level. Recall that the expected number of errors is $7\alpha$ in both designs.

We simulated these designs 10000 times. Design 1 had a Type I error frequency of $0.05104$ while Design 2 had a false rejection rate of $0.05124$. 
In Design 1, the test statistics had high pairwise correlation whereas the tests in Design 2 were independent (Fig.~\ref{fig:reused_control_correlated_test_statistics}). The distribution of errors was markedly different between Design 1 and Design 2, as shown in Table \ref{tab:reused_control_error_distribution}. While Design 1 had an overall lower rate of familywise error, Design 2 is strictly less risky in the stop-loss order. (Fig.~\ref{fig:reused_control_stoploss}). This has an important epistemic interpretation: in the event that subsequent replication efforts testing of $T_1$ vs. $C$ indicates that $H_{01}$ is correct, then there is a strong probability that at least one of the other comparisons is incorrect since the conditional probability of $\geq 2$ errors given an erroneous rejection $R_1$ is much higher in Design 1 than it is in Design 2 under the global null.

\begin{figure}[htbp]\label{fig:reused_control_correlated_test_statistics}
    \centering
    \includegraphics[width=0.8\textwidth]{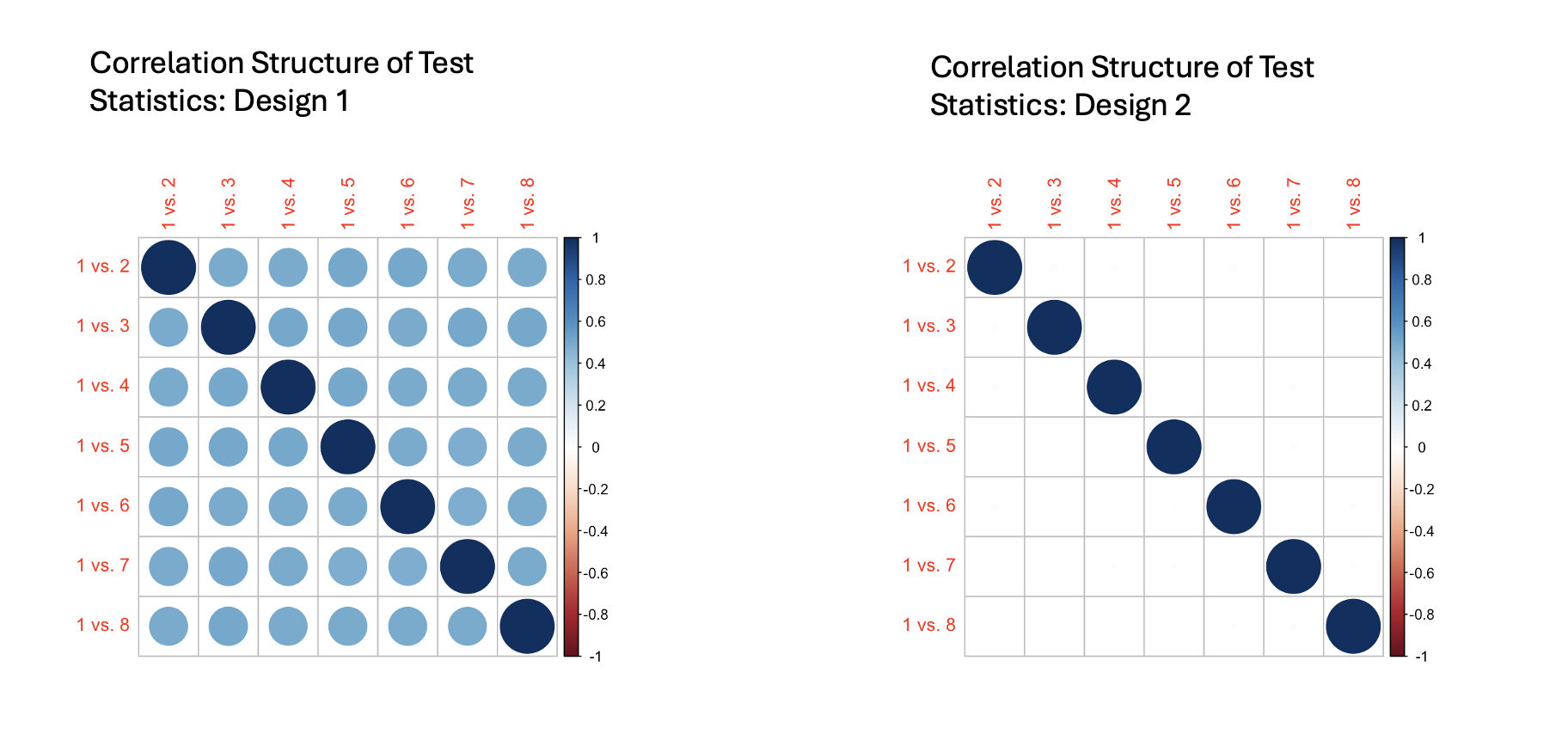}
    \caption{Pairwise Dependence in Test Statistics by Design}
    \label{fig:distribution_of_errors}
    \end{figure}

\begin{table}[h!]\label{tab:reused_control_error_distribution}
\centering
\begin{tabular}{c|c c|c c}
\toprule
\multirow{2}{*}{Number of False Rejections} & \multicolumn{2}{c|}{Design 1} & \multicolumn{2}{c}{Design 2} \\ \cline{2-5}
 & N & Frequency & N & Frequency \\ 
\midrule
0 & 7682 & 0.7682 & 6896 & 0.6896 \\ 
1 & 1558 & 0.1558 & 2661 & 0.2661 \\ 
2 & 468  & 0.0468 & 404  & 0.0404 \\ 
3 & 164  & 0.0164 & 38   & 0.0038 \\ 
4 & 74   & 0.0074 & 1    & 0.0001 \\ 
5 & 39   & 0.0039 & -    & -      \\ 
6 & 9    & 0.0009 & -    & -      \\ 
7 & 6    & 0.0006 & -    & -      \\ 
\bottomrule
\end{tabular}
\caption{Comparison of Distribution of Error by Design.}
\end{table}

\begin{figure}[htbp]\label{fig:reused_control_stoploss}
    \centering
    \includegraphics[width=0.8\textwidth]{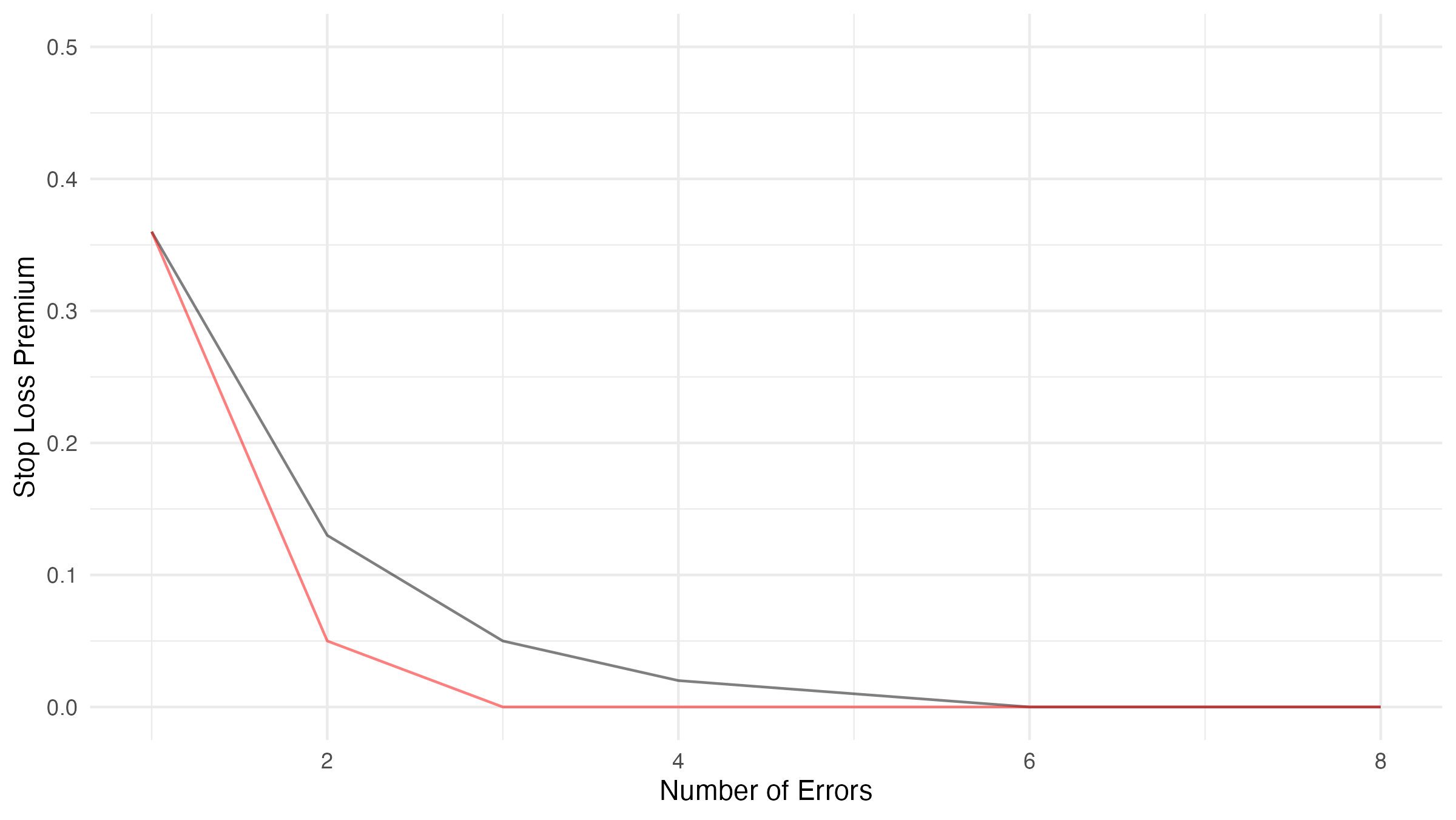}
    \caption{Stop Loss Premiums For Design 1 (Black) vs. Design 2 (Red)}
    \label{fig:distribution_of_errors}
    \end{figure}
\end{example}

For our next example, we consider an example where a difference in survival up to $1$ year was performed, and then a subsequent study using the same data was performed evaluating survival difference at $5$ years. While the endpoints are different, the testing procedures are still dependent. This is an example of compounding error through a chain of incorrect research, wherein a Type I error occurs (e.g. a between-group survival difference at $1$ year) and then related hypotheses are formed and then reassessed using the same observational units (e.g. a between-group survival difference at $5$ years).

\begin{example}
Consider two groups $A$ and $B$ with $100$ patients in each. Assume that the survival distributions of $A$ and $B$ are the same; namely, that they are generated by a Weibull distribution with mean survival of $2.5$ years and scale parameter $2$. We tested for the difference in survival truncated at 1 year and at 5 years using the log-rank test. Let $R_1$ be the event that rejects at 1 year, and $R_5$ be the event that rejects at the 5 year mark.

In our simulation of 10,000 instances of this design, the contingency table of errors was as follows:

\begin{table}[h]
    \centering
    \caption{Contingency Table for Rejection of Hypothesis at 1 and 5 Years}
    \begin{tabular}{c|ccc}
          &   $\neg R_1$ &   $R_1$ \\
        \midrule
        $\neg R_5$ & 9406 &  77 \\
    $R_5$ &   86 & 431 \\
    \end{tabular}
    \label{tab:contingency_table}
\end{table}

In this simulation, the frequency of a Type I error at $5$ years given a Type I error at $1$ year was $\frac{431}{431+77} = 84.8\%$. The Pearson correlation between the log-rank test statistics was $0.9726$. In this example, a propagation of errors occurred due to the fact that survival information at $1$ year provided considerable information leading to dependent testing.

One possible design that eliminates the dependence in Type I errors under the null is to randomly sort the cohort into two groups and perform a gatekeeping procedure: if $R_1$ is rejected in group $A$, then test $R_5$ in group $B$. By testing $R_5$ within group $B$, the test statistics, and consequently the Type I error rates, will be independent.
\end{example}

\subsection{Limitations of Existing Measures of Type I Error Control}

We now discuss how common measures of Type I error control relate to risk preferences regarding the distribution of error. While typically restricted to measuring Type I error rates, these measures are well-defined for any sequence of error events $E_1,\ldots, E_n$. In this context, we examine the extent to which management of common measures of Type I error rates relate to risk-averse decisionmaking. In this setting, we can completely describe the distribution of the count of errors in the two-study portfolio $\mathcal{P} = \{ E_1,E_2 \}$. Assume that $\mathrm{pr}(E_i) = \alpha$.

First, observe that regardless of the dependence between $E_1$ and $E_2$, the $PCER$ is $\alpha$. Thus, $PCER$ is insensitive to the correlational structure of the testing procedure.

\begin{proof}
By the linearity of expectation, PCER can be rewritten as 
\begin{equation}
PCER(\mathcal{P}) = \frac{1}{m}\sum\mathbb{E}\left[\mathbb{1}_{E_i} \right].
\end{equation}
Suppose that each testing procedure is performed at level $\alpha$. Then 
\begin{equation}
\mathbb{E}\left[\mathbb{1}_{E_i}\right] = \alpha
\end{equation}
and so 
\begin{equation}
PCER(\mathcal{P}) = \frac{m}{m} \alpha = \alpha.
\end{equation}
\end{proof}

An immediate consequence of this is that the status quo policy of considering a result statistically significant so long as $\mathrm{pr}(E) \leq \alpha$ is neutral to the dependence structure of the events $E_i$, an attitude which may worry a risk-averse investigator.

By contrast, the FWER provides one measure of the spread of the distribution of errors, since 
\begin{equation}
FWER(\mathcal{P}) = \mathrm{pr}\left(\sum\mathbb{1}_{E_i} \geq 1\right).
\end{equation}

In the context of the two-event portfolio $\mathcal{P}$, this reduces to
\begin{equation}
FWER(\mathcal{P}) = \mathrm{pr}(E_2) + \mathrm{pr}(E_1)(1-\mathrm{pr}(E_2|E_1))
\end{equation}
and is thus minimized when $\mathrm{pr}(E_2|E_1) = 1$ at value $\alpha$ and maximized at $2\alpha$ when $\mathrm{pr}(E_2|E_1) = 0$. On that basis alone one might conclude that tests with positive dependence---those for which $\mathrm{pr}(E_2|E_1) > \mathrm{pr}(E_2)$---are desirable. However, the reduced familywise error rate comes at the cost of shifting the distribution of the count of errors. If one were to focus only on FWER, one might think that positively-dependent testing procedures are unequivocally more desirable than independent testing procedures.

The False Discovery Rate is plagued by the same preference as Family-wise Error Rates towards positively correlated tests. Since only Type I errors are penalized by $FDR$, the worst case expectation occurs when the global null $\bigcap H_{0i}$ holds. In that case all rejections are false positives, so the random variable $\frac{\# \text{Type I Errors}}{\# \text{Rejections}}$ is a Bernoulli random variable which is $0$ if there is no error and $1$ otherwise. Consequently, under the global null,
\begin{align}
    FDR(\mathcal{P}) & = \mathbb{E}\left[\frac{\# \text{Type I Errors}}{\# \text{Rejections}} \right] \\
    & = \mathbb{P}(E_1\cup E_2) \\
    & = FWER(\mathcal{P}).
\end{align}
Thus, under the global null, $FDR(\mathcal{P})$ and is minimized under perfectly positively-dependent testing procedures.

\subsection{Case Study II: Type II Errors under the Global Alternative and Fixed Level $\alpha$}\label{sec:power_tradeoff}

In this section, we consider a paradigmatic example that might face an investigator. Suppose that the level $\alpha$ is fixed (for example, $\alpha = 0.05$). The investigator for study $S$ may determine the Type II error rate of their study under a specific effect size $\Delta_0$ and known sample size (vector) $\overline{n}$:
\begin{equation}
\pi_{S}(\alpha, \Delta_0,\overline{n}).
\end{equation}
On one hand, the investigator wants to ensure adequate power for his or her study by using a large amount of data; on the other hand, we have seen how data reuse can result in dependent error events.

In Section \ref{sec:typeI_error} we discussed how investigators could reduce sample size while retaining level $\alpha$ error control to reduce or fully eliminate positive dependence in the testing procedures that gave rise to riskier inferential portfolios. In that setting, we could exhibit portfolios $\mathcal{P}_1$ and $\mathcal{P}_2$ with identical expected counts of Type I errors under the null such that the testing procedure $\mathcal{P}_1$ with independent tests was preferred to $\mathcal{P}_2$ with dependent tests by ensuring that the studies in $\mathcal{P}_1$ were using disjoint sets of data.

In the current setting of Type II errors using tests of a fixed level $\alpha$, however, there is generally a tradeoff between power and sample size. In general, portfolios which have different means are not stop-loss orderable, meaning that the preference between $\mathcal{P}_1$ and $\mathcal{P}_2$ depends on the specific utility function of the risk-averse agent in question.

For the remainder of the section we assume for simplicity that across the portfolio of studies, the alternative $H_{1i}: \theta_i = \Delta_{0i}$ is a point hypothesis, and that the global alternative hypothesis is true. 
In Section \ref{sec:portfolio_optimization} we will see how an expected utility approach can accommodate more realistic scenarios.

Our first observation is that depending on the data available to the investigator, reduction in sample size to avoid data reuse generally results in a decrease in power.

Let $S_i(\overline{n})$ be the study using test $T_i$ with sample vector $\overline{n}$. Assume that Type II error rate given $\theta\in \Theta$ is given by 
\begin{equation}
\pi_{S_i}(\alpha,\theta, \overline{n}).
\end{equation}

Then under the global alternative, the expected count of errors for the portfolio $\mathcal{P}(\overline{n}_1,\ldots, \overline{n}_k) = \{S_1(\overline{n}_1),\ldots, S_k(\overline{n}_k)\}$ depending on sample size vectors $\overline{n}_1,\ldots, \overline{n}_k$ satisfies
\begin{equation}
\mathbb{E}[u(\mathcal{P}(\overline{n}_1,\ldots, \overline{n}_k))] = -\int\left( \sum_{i} \pi_{S_i}(\alpha,\theta, \overline{n}_i) p(\theta)\right)
\end{equation}
where $p(\theta)$ is a prior probability density over $\theta$.

Fixing $\alpha$ and assuming the alternatives $H_{0i}: \theta = \Delta_{0i}$ this reduces to
\begin{equation}
\mathbb{E}[u(\mathcal{P}(\overline{n}_1,\ldots, \overline{n}_k))] = -\sum_{i} \pi_{S_i}(\alpha,\Delta_{0i}, \overline{n}_i).
\end{equation}

Assuming that the Type II error rate function $\pi$ is strictly monotonic in $\overline{n}$, reducing sample size implies an increase in the incidence of Type II errors across the portfolio: if each $\overline{m}_i < \overline{n_i}$ then 
\begin{equation}
\mathbb{E}[u(\mathcal{P}(\overline{m}_1,\ldots, \overline{m}_k))] \leq \mathbb{E}[u(\mathcal{P}(\overline{n}_1,\ldots, \overline{n}_k))].
\end{equation}

The kinds of large-scale datasets we have in mind routinely have tens of thousands---if not hundreds of thousands---of data points. Power is consequently not hard to come by, and the tradeoff between ensuring independence of testing procedures versus ensuring adequate power to detect a meaningful effect can be very small.

\begin{example}
Suppose that we are comparing a control group $C$ to two treatment groups $T_1$ and $T_2$. Assume that the alternative hypothesis is given by 
$C \sim \mathcal{N}(0,1)$ and $T_1, T_2 \sim \mathcal{N}(0.5,1)$. Let $\Delta_{01} = \Delta_{02} = 0.5$ be the effect size for the alternative for both studies.

Suppose that there are $|C| = 1000$ control units and $|T_1| = |T_2| = 500$ treatment units in each arm.

Under the global alternative hypothesis, the Type II error rate for the two-sample $t$-test assuming each the $500$ unit treatment groups are compared to the $1000$ unit control group is $< 10^{-10}$. Thus, the expected count of Type II errors of this portfolio is
\begin{equation}
0\geq \mathbb{E}[\{S_1(1000,500),S_2(1000,500)\}] \geq -2\times10^{-10}.
\end{equation}

If, however, we partition the control group into $500$ units to compare to $T_1$ and $500$ to compare to $T_2$, the Type II error rate is $\approx 1.4\times 10^{-9}$ per study. Thus, 
\begin{equation}
0\geq \mathbb{E}[\{S_1(500,500),S_2(500,500)\}] \geq -3\times10^{-9}.
\end{equation}

In the latter design, we have ensured that the testing procedures are independent at the cost of exposure to $\leq 3\times10^{-9}$ more Type II errors.
\end{example}

However, a key use case for large-scale databases is for the analysis of rare subgroups for which recruitment into randomized control trials can be exceedingly difficult. In these circumstances, the tradeoff between reducing testing dependence and achieving power can be intolerably high.

\begin{example}\label{ex:low_sample_tradeoff}
Assume that we are comparing a control group $C$ to two treatment groups $T_1$ and $T_2$.

In our first dataset $D_1$, we have $n_C = 100$ control units, and $n_1 = n_2 = 50$ treatment units each. We assume that each of these groups is drawn from a normal distribution, where 
$C \sim \mathcal{N}(0,1)$ and $T_1, T_2 \sim \mathcal{N}(0.5,1)$. We aim to compare performing the $t$-test at $\alpha = 0.05$ under two data usage scenarios:
\begin{enumerate}
\item Data gluttony: reuse the full control group $C$ in both comparisons, and
\item Reuse avoidant: compare $T_1$ and $T_2$ each to disjoint groups of $50$ units in $C$.
\end{enumerate}

Data gluttony in this setting is justified by seeking to minimize the average rate of Type II errors, while reuse avoidance is seeking to eliminate the dependence on testing procedures.

Under the data gluttony strategy, the Type II error rate under the alternative specified by $(\mu_i,\sigma_i) = (0.5,1)$ for $i = 1,2$ vs. $(\mu_0,\sigma_0) = (0,1)$ is $\approx 18.2\%$. By contrast, the Type II error rate under the reuse avoidance strategy is $\approx 30.3\%$.

Thus, the expected count of errors in the two-study portfolio is $0.364$ under the data gluttony strategy and is $0.606$ under the reuse avoidance strategy, a $66.5\%$ increase in portfolio disutility.
\end{example}

The above example shows how, in small sample settings, adopting a data reuse minimizing approach can dramatically reduce the portfolio utility under the alternative. However, for sufficiently large samples the tradeoff between the expected count of errors and the dependence between errors diminishes.

\subsection{Choice of Utility Function and Portfolio Optimization}\label{sec:portfolio_optimization}

\subsection{The Utility of a Study}
The Expected Utility Theory of von Neumann and Morgenstern gives a necessary and sufficient condition for agents to rationally choose between portfolios on the basis of their utility function \cite{mas1995microeconomic}.

While much research is conducted using rule-of-thumb values for the level $\alpha$ (most ubiquitously $\alpha = 0.05)$, in the context of inferential tasks, this approach has been applied to balancing the level $\alpha$ and power $1-\beta$. Since power has diminishing marginal returns in $\alpha$, for a fixed sample size one can optimize the choice of $\alpha$ to minimize the expected count of inferential errors given a prior over the parameter space. Discussions of identifying the optimal balance between power and significance have been ongoing since the advent of null hypothesis statistical testing \cite{lehmann_uniformly_2022} and have recently resurfaced in debates about significance testing in psychology (\cite{lakens_justify_2018}).

However, seeking to avoid inferential errors across a portfolio is insufficient as an account of conducting scientific studies. Case-in-point: suppose that the utility of all studies have utility $u$ bounded above by $0$ and expected negative utility (as in the error avoidance utility function of a portfolio) and that the utility of a portfolio is the sum of the utility of each study. Then for any portfolio $\mathcal{S} = \{S_1,\ldots, S_n\}$ of studies and new study $\{S_{n+1}\}$, adding $S_{n+1}$ to the portfolio would be strictly worse since 
\begin{equation}
u(\mathcal{S}\cup\{S_{n+1}\}) = u(\mathcal{S}) + u(S_{n+1}) \leq u(\mathcal{S})
\end{equation}
and, in particular,
\begin{equation}
\mathbb{E}[u(\mathcal{S}\cup S_{n+1})] \leq \mathbb{E}[u(\mathcal{S})].
\end{equation}
Thus, under such a utility function a rational agent is never justified in conducting an additional study. To address this problem, analysis of some anticipable social utility can be constructed.

\begin{example}
One choice of utility function would reflect some consideration of the impact of the research to improve social welfare vis-à-vis an analysis of the expected annual change in quality-adjusted life years (QALYs) one would expected as a function of the true parameter $\theta$ and the statistical decision reached by the study. For example, if when comparing a novel treatment $T_1$ against an existing gold standard $T_0$ of equal cost with $\theta = ATT(Y|T) = \mathbb{E}[Y|do(T_1)]-E[Y|do(T_0)]$ the average difference in QALYs among the treated, a candidate utility function might be given by 
\begin{equation}
u(\theta, \text{decision}) = \begin{cases}
\theta & \text{if $H_0: \theta < 0$ is rejected } \\ 
\min(-\theta,0) & \text{if $H_0: \theta < 0$ is not rejected}.
\end{cases}
\end{equation}
where it is assumed that the null $H_0:\theta < 0$ is rejected, all patients receiving $T_0$ will receive $T_1$ instead; and if $H_0:\theta = 0$ is not rejected then all patients will continue to receive treatment $T_0$. In these cases, the social utility still depends on the true value of the parameter $\theta$: if we falsely reject the one-sided $H_0$ then negative utility would be generated by converting all patients to $T_1$, while if we fail to reject $H_0$ but $\theta > 0$ then we have failed to improve the QALYs of patients by continuing to assign patients to the status quo $T_0$.
\end{example}

With all this being said, the choice of utility function can be fraught, especially in the context of a committee made up of various stakeholders with different values. In such a context, having a measure of inferential downside in isolation may be useful. This can be especially true of basic or speculative studies, which may have long-term upside in societal utility that may be hard to anticipate or quantify. In such a case, the decision study may be permitted to be performed on the data but its $\alpha$ and $\beta$ optimized according to some function of the count of inferential error. In this context, the utility function might be a nonlinear transformation of the count of errors that would reflect an explicitly risk-averse attitude.

\begin{example}
For example, the utility function might be $u(\mathcal{P}) = -(\sum (\#E_I + \#E_2))^2$ penalizing compounding inferential errors is risk-averse, as it assigns greater disutility to a positively-dependent tests.
\end{example}

\subsection{Performing Portfolio Optimization}
Given a utility function $u$ and set of variable parameters $z$ associated with a portfolio $\mathcal{P} = \mathcal{P}(z)$ and a feasible set $Z$, one can pose the optimization problem
\begin{equation}
z \in \text{argmax}_{z\in Z} \mathbb{E}[u(\mathcal{P}(z))].
\end{equation}

Techniques for analyzing this optimization problem are well known in finance (for example, \cite{joshi_introduction_2013}) and actuarial science (\cite{kaas_modern_2008}). In this section we discuss only grid-search based approximation of optimal portfolios to preserve generality.

Applying this framework to this topic, the inputs required to perform portfolio optimization are:
\begin{enumerate}
\item A fixed list of hypotheses to evaluate $\{H_{0i}\}_{i\leq n}$,
\item Families of testing procedures varying in sample size vector for each hypothesis,
\item A utility function $u$ which is a function of the statistical decisions (reject/fail to reject) and perhaps the parameter $\theta$,
\item A joint prior distribution for the parameters $\theta$ so that the expected utility of a portfolio can be calculated, and 
\item The sample size vector of $D$,
\item A description of the subsampling procedures, such as the vector of per-study fractions of data $\overline{r} = (r_1,\ldots, r_n)$ in $D$ to sample without replacement.
\end{enumerate}

At the level of multiple studies, the utility-maximizing choices of levels $\alpha_i$ and $\overline{r}$ can be approximated using a grid search technique.

The challenges with this approach are primarily logistical, not computational: it requires the knowledge of a fixed number of hypotheses of interest, an analysis plan for each, the stipulation of a utility function, and a joint prior over all the effects simultaneously. While some datasets may be well-poised to conduct a portfolio optimization exercise due to their data governance and control over which studies are approved, datasets not managed in such a manner are not suitable for this approach.

To elaborate on Step 6 above, in the following section we discuss a pragmatic approach to data usage that individual investigators can perform to reduce their data usage while maintaining tolerable control over inferential errors.

\section{Subsampling and Data Temperance}

In this section we describe a pragmatic complement to data gluttony that we term data temperance. Data temperance can be understood as the policy that identifies the minimum amount of data necessary to perform the inferential task in question and then subsamples from the dataset $D$ a sample of the requisite size.

Given the epistemic risk posed by dependent testing procedures, it may surprise some that we do not require the subsamples to be mutually disjoint. In the absence of resource constraints, performing statistical evaluations of hypotheses on disjoint samples would be preferable. However, requiring disjointness poses serious challenges to the investigator.

Most practically, it poses a coordination problem between independent researchers: if the data used must be partitioned amongst investigators, the data must be both allocated and verified to be disjoint across investigators. For large datasets accessible by many investigators simultaneously, this is a major challenge. Similarly, partitioning gets harder when there are analyses directed at small subgroups.

Second, partitioning the data can incentivize researchers to propose multiple studies as quickly as possible to box out other researchers from using the database.

Moreover, partitioning may simply be an inefficient use of data insofar as it permits a relatively scarce number of studies. Suppose that we have a set of hypotheses each requiring a sample of size $\geq k$ using dataset $D$ of size $n$. Then one can perform at most $C \leq \lfloor \frac{n}{k} \rfloor$ many studies while maintaining disjointness. In the following sections we will discuss lower bounds for the number of subsamples one can draw of size $k$ while maintaining a low probability of a large pairwise intersection in datasets. This can serve as a pragmatic measure of the capacity of a dataset to accommodate boundedly-dependent inferential tasks. While independence is, on balance, preferable to positive dependence, we are seeking pragmatic and consequently imperfect techniques to lower our epistemic risk.

\subsection{The Investigator's Workflow}

The above sections illustrated how the dependence on testing procedures can be eliminated by partitioning the dataset across studies. However, for many datasets the allocation of data to individual studies that cannot be subsequently reused might be undesirable. For example: (i) it may permit few studies to be performed using the data, (ii) one research group may hog data by requiring a high degree of power, limiting the power of other studies, (iii) the desire to preserve the dataset for the evaluation of heretofore uncontemplated analyses.

One avenue for reducing data reuse while not allocating portions of the dataset for the exclusive use of a single inferential task is to allow for investigators to subsample the data. The basic workflow is this:

\begin{enumerate}
\item The investigator specifies the inclusion/exclusion criteria of the study, Type I error rate $\alpha$, a minimum relevant effect size $\Delta_0$, and Type II error rate $1-\beta$ depending on $\alpha$ and $\Delta_0$.
\item Using these inputs, the researcher determines a required sample size vector $\overline{n}_0$.
\item If the dataset $D$ has the data required to perform the analysis, the investigator samples from $D$ without replacement to obtain a sample $\widehat{D}$.
\item The investigator performs the desired analysis on $\widehat{D}$.
\end{enumerate}

An approach like this has been advocated by \cite{dahl_data_2008}), and data splitting or data reduction is ubiquitous in machine learning.

While this approach does not ensure nonempty overlap in data, it does at least mitigate its effect.

\subsection{Why Subsampling Without Replacement is Justified}

In the above workflow, the recommendation was to subsample from $D$ to a sample size sufficient for ensuring adequate power $\beta$ and level $\alpha$. Why is this justified in an inferential context? Assuming that the original sampling procedure is exchangeable for all sample sizes $n$, subsampling $k<n$ without replacement of them from the sample is equidistributed to the equivalent to sampling $k$ elements under the original sampling procedure.

\begin{proposition}\label{prop:subsample_justification}
Let $\Omega$ be a sample space and $\varsigma_{\infty}$ be a sampling procedure on $\Omega$ such that:
\begin{enumerate}
\item $\varsigma_{\infty} = (X_1,\ldots, X_n,\ldots)$ is an infinite sequence of elements in $\Omega$,
\item $\varsigma_n = (X_1,\ldots X_n)$,
\item $\varsigma_{\infty}$ is finitely exchangeable; that is, for all finite $n$ and permutations $\phi: \{1,\ldots, n\} \to \{1,\ldots, n\}$ the truncations of $\varsigma_{\infty}$ to the first $n$ draws satisfy
\begin{equation}
(X_1,\ldots, X_n) \sim (X_{\phi(1)},\ldots, X_{\phi(n)}).
\end{equation}
\end{enumerate}
Let $\nu_{k,n}$ with $k\leq n$ be the sampling procedure given by drawing elements from $\varsigma_{n}$ and then sampling $k$ uniformly without replacement from $\varsigma_n$.

Then for all $n$ and all $k\leq n$
\begin{equation}
\varsigma_k \sim \nu_{k,n}.
\end{equation}
\end{proposition}

\begin{proof}
Suppose we sample according to the scheme $\nu_{k,n}$. Let $(X_1,\ldots, X_n)$ be the sample of size $n$ drawn from $\varsigma_n$. Let $(Y_1,\ldots, Y_k)$ be sampled from $\nu_{k,n}$. Sampling $k$ elements without replacement from a set of size $n$ is equivalent to uniformly drawing an injective function $\phi:\{1,\ldots, k\}$, so that each $Y_i = X_{\phi(i)}$. Let $\widetilde{\phi}: \{k+1,\ldots, n\} \to (\{1,\ldots, n\}\setminus \text{im}(\phi))$ extend $\phi$ to a permutation of $\{1,\ldots, n\}$.

By exchangeability, 
\begin{equation}
(X_1,\ldots, X_n) \sim (X_{\phi(1)},\ldots, X_{\phi(k)}, X_{\widetilde{\phi}(k+1)},\ldots, X_{\widetilde{\phi}(n)})
\end{equation}
and
\begin{equation}
(X_{\phi(1)},\ldots, X_{\phi(k)}) \sim \varsigma_k.
\end{equation}
Thus, $\nu_{k,n} \sim \varsigma_k$ as desired.
\end{proof}

The content of this proposition is that so long as your inferential procedure relies on an assumption of an exchangeable sampling procedure for each sample size $n$, sampling without replacement $k$ elements from a larger dataset is inferentially equivalent to having sampled $k$ elements according to the original sampling procedure. In particular, a sequence of iid draws of a distribution $\theta$ satisfies the conditions of Proposition \ref{prop:subsample_justification}.

Consequently, the Type I error rates and Type II error rates for the subsampled data will be distributionally equivalent to data sampled directly from $\Omega$ of the same size under exchangeable sampling $S$.

\subsection{When is the Dataset Used Up? Extremal Aspects of Data Reuse Scaling in the Number of Studies}

A natural concern is how the subsampling procedure mitigates the risk of high degrees of pairwise data overlap which induce pairwise dependent testing procedures. To study this, we define the following measure of worst-case risk of sampling $C$ subsets of size $k$ from a set of size $n$ and realizing an intersection of size $\geq \ell$:
\begin{equation}
\mathrm{pr}\left(\max_{1\leq i<j\leq C} |X_i \cap X_j| \geq \ell \right).
\end{equation}
For example, if we aim to have a $\leq 5\%$ chance of a pairwise intersection of size $\geq \frac{k}{20}$ we would demand 
\begin{equation}
\mathrm{pr}\left(\max_{1\leq i<j\leq C} |X_i \cap X_j| \geq \frac{k}{20} \right) \leq 5\%.
\end{equation}

For fixed $N$, $k$, and $\ell \leq k$ the probability is nondecreasing in the number $C$ of subsets drawn. We probe this dependence through combinatorial estimates.

\subsection{Combinatorial Upper Bounds for Low-Data Usage}

First observe that the sizes $|X_i\cap X_j|$ are not independent for $i\neq j$. Consequently we cannot appeal to independence in the construction of our bounds. However, we can reduce the problem by iterative uses of the union bound.

Let $n>0$ and let $[n] = \{0,\ldots,n-1\}$. We are uniformly sampling $C>1$ $k$-subsets with replacement $X_1,\ldots,X_C$ of $[n]$ and are interested in getting an upper bound on the probability that the maximum size of the pairwise intersections $X_i\cap X_j$ for $i\neq j$ is bounded above by some fixed $\ell< k$. More precisely, we want to get an expression for
\begin{equation}
\mathrm{pr}\left(\max_{1\leq i<j \leq C} |X_i\cap X_j| \geq \ell \right).
\end{equation}

Observe that the event in question is the disjunction between $\binom{C}{2}$ many inequalities:
\begin{equation}
\max_{1\leq i<j \leq C} |X_i\cap X_j| \geq \ell \iff \bigvee_{1\leq i < j \leq C} |X_i\cap X_j| \geq \ell.
\end{equation}
By the uniformity of sampling,
\begin{equation}
\mathrm{pr}(|X_i\cap X_j|\geq \ell) = \mathrm{pr}(|X_1\cap X_2|\geq \ell).
\end{equation}
So, by the union bound we can conclude that
\begin{equation}
\mathrm{pr}\left(\max_{1\leq i<j \leq C} |X_i\cap X_j| \geq \ell \right) \leq \binom{C}{2}\mathrm{pr}(|X_1\cap X_2|\geq \ell).
\end{equation}
This admits a further simplification: by the uniformity assumption on sampling, the random variables $|X_1 \cap X_2|$ and $|X_1\cap[k]|$ are identically distributed, with the latter being $|X_1\cap[k]| \sim \text{Hypergeometric}(n,k,k)$.

Thus, a sufficient condition on $C$ to ensure a $\leq p_{tol}$ probability of an $\geq \ell$-sized intersection is 
\begin{equation}\label{eq:max_c_estimate}
C \leq \sqrt{\frac{2 p_{tol}}{\mathrm{pr}(|X_1\cap [k]|\geq \ell)}}.
\end{equation}
It remains to estimate the quantities $\mathrm{pr}(|X_1\cap X_2| \geq \ell)$.

With computer assistance one can directly evaluate the quantity $\mathrm{pr}(|X_1\cap [k]|\geq \ell)$, but we give a less sharp upper bound that demonstrates the at-least exponential rate of growth in the capacity to sample sets with low probabilities of large intersections.

\begin{proposition}\label{prop:concentration_bound} 
Let $n > k = |X_1|$, $\ell \in \left[\frac{k^2}{n},k\right]$, and
\begin{equation}
b = \frac{\ell}{k} - \frac{k}{n}.
\end{equation}
Then 
\begin{equation}
\mathrm{pr}(|X_1\cap [k]|\geq \ell) \leq e^{-2kb^2}.
\end{equation}
\end{proposition}

\begin{proof}
The hypergeometric random variable $|X_1\cap[k]| \sim \text{Hypergeometric}(n,k,k)$ satisfies the following tail bound (see, e.g., \cite{chvatal_tail_1979}or \cite{skala_hypergeometric_2013}) for $0 \leq b \leq \frac{k}{n}$:
\begin{equation}
\begin{array}{ccc}
\mathrm{pr}\left(|X_1\cap[k]| \geq \mathbb{E}[|X_1\cap[k]|] + bk\right) & = & \mathrm{pr}\left(|X_1\cap[k]| \geq \frac{k^2}{n}+ bk\right) \\
& \leq & e^{-2kb^2}.
\end{array}
\end{equation}
Writing $\ell = \frac{k^2}{n}+ bk$ we isolate $b = \frac{\ell}{k} - \frac{k}{n}$. Thus 
\begin{equation}
\mathrm{pr}(|X_1\cap [k]|\geq \ell) \leq e^{-2kb^2}.
\end{equation}
\end{proof}

This immediately yields a bound on the number of samples of size $k$ from a set of size $n$ while controlling the probability of a large pairwise intersection from occurring.

\begin{corollary}\label{cor:hypergeometric_tail} 
Let $r_1 = \frac{k}{n}$ and $r_2= \frac{\ell}{k}$ and $p_{tol}$ be fixed. Assume $r_2 \in \left[\frac{k}{n},1\right]$. Then the quantity 
\begin{equation}
C_0(k;r_1,r_2,p_{tol}) = \text{argmax}_A \left[\mathrm{pr}\left(\max_{1\leq i<j\leq A} |X_i \cap X_j| \geq kr_2 \right) \leq p_{tol}\right]
\end{equation}
where $X_i,X_j$ are independently and uniformly sampled $k$-subsets of a set of size $n$ satisfies
\begin{equation}
C_0(k) \geq \sqrt{2p_{tol}} e^{k \left| r_1 - r_2 \right|^2}.
\end{equation}
\end{corollary}

\begin{proof}
By (\ref{eq:max_c_estimate}) we know that $C_0(k) \geq \sqrt{\frac{2 p_{tol}}{\mathrm{pr}(|X_1\cap [k]|\geq \ell)}}$. By Proposition \ref{prop:concentration_bound} we know that 
\begin{equation}
\mathrm{pr}(|X_1\cap X_2|\geq kr_2) \leq e^{-2k \left| r_1 - r_2 \right|^2}.
\end{equation} 
Hence 
\begin{equation}
\begin{array}{ccc}
C_0(k) & \geq & \sqrt{2p_{tol}}\sqrt{\frac{1}{\mathrm{pr}(|X_1\cap [k]|\geq \ell)}} \\
& \geq & \sqrt{2p_{tol}} \sqrt{\frac{1}{e^{-2k \left| r_1 - r_2 \right|^2}}} \\
& = & \sqrt{2p_{tol}} e^{k \left| r_1 - r_2 \right|^2}.
\end{array}
\end{equation}
as desired.
\end{proof}

To apply these bounds, one must first verify that $\ell \geq \frac{k^2}{n}$ (equivalently $\frac{\ell}{k}\geq \frac{k}{n}$).

\begin{example}
Let $n=10,000$ and suppose that we want to sample without replacement $k = 2,000$ observations per study. How many studies can we draw while maintaining a $\leq 5\% = p_{tol}$ probability of a large intersection for a given $\ell$?

To apply Corollary \ref{cor:hypergeometric_tail}, we have to at least tolerate $\ell \geq \frac{k^2}{n} = 400$, which means that $\frac{\ell}{k} \geq 0.2$. As Table \ref{tab:max_c_table}, once we allow for $\frac{\ell}{k} \geq 0.275$, over $24,000$ studies can be performed while controlling the rate of a maximal pairwise intersection.

\tbl{Relation of $C_0$ to the tolerance for the maximal intersection size}{
\begin{tabular}{ccc}
\textbf{$\frac{\ell}{k}$} & \textbf{$\ell$} & \textbf{$C_0$ lower bound} \\
0.200 & 400 & 0 \\
0.225 & 450 & 1 \\
0.250 & 500 & 46 \\
0.275 & 550 & 24,311 \\
0.300 & 600 & $1.53 \times 10^{8}$ \\
0.325 & 650 & $1.17 \times 10^{13}$ \\
0.350 & 700 & $1.10 \times 10^{19}$ \\
0.375 & 750 & $1.26 \times 10^{26}$ \\
0.400 & 800 & $1.75 \times 10^{34}$ \\
0.425 & 850 & $2.96 \times 10^{43}$ \\
\end{tabular}}
\label{tab:max_c_table}

\end{example}

Already on the order of $n=10,000$ a very large number of studies using $\frac{n}{5}$ observations can be performed with low risk of intersections of size $\geq \frac{k}{3}$. Since the power and/or estimate precision of a study are generally concave in sample size, as $n$ grows typical studies will need only a very sparse subset of the observations to achieve strong power. For example, if $n = 1,000,000$, $k = 10,000$, and $\ell = 500$ (so that $\frac{\ell}{k} = \frac{1}{20}$) and $p_{tol} = 5\%$, up to $2,810,034$ subsets could be safely sampled.

\subsection{Guaranteed Rates of Data Reuse for Data-Intensive Studies}

The amount of data contained in the database and data intensiveness of the tasks using the data can inform how many studies it might take to result in large amounts data reuse.

By contrast with the previous section, this section addresses the question of when overlap is guaranteed. These results are most interesting for small, shared datasets and pose mathematical limits on how data can be reused.

More formally, let $D$ be a dataset of size $N$, $k<N$ be the per-study required sample, $\ell \leq k$ be an integer, and $p\in [0,1]$ be probability. We define $C = C(N,k,\ell,p)$ to be the minimum $C$ such that independent uniform samples $X_1,\ldots, X_C$ of $D$ of size $k$ satisfy
\begin{equation}
\mathrm{pr}\left(\max_{1\leq i<j\leq n} |X_i \cap X_j| \geq \ell \right) \geq p.
\end{equation}

We consider two examples.

\begin{example}
In the most simple case, we investigate $\ell$ when $C = 2$, $p=1$, and $k \geq \frac{N}{2}$.

In the most simple case, significant data reuse can be guaranteed even in the case of two data-intensive studies. Suppose that we have data $D$ of size $N$, and two studies each require $k < N$ data points $X_1, X_2\subseteq D$. The studies each sample a random subset of size $k$ from $D$.

By the inclusion-exclusion principle, we know that 
\begin{equation}
|X_1\cap X_2| = |X_1|+|X_2|-|X_1\cup X_2|.
\end{equation}
Since $X_1\cup X_2 \subseteq D$, $|X_1\cup X_2| \leq N$ and so
\begin{equation}
|X_1\cap X_2| \geq 2k - N.
\end{equation}
Consequently, if $k>\frac{N}{2}$ then $X_1\cap X_2$ is nonempty.

Supposing that we want to avoid rates of data overlap above a given $\lambda = \frac{\ell} {k}\in [0,1]$: that is, $|X_1\cap X_2| \geq \lambda k$. If $k\geq \frac{N}{2-\lambda}$ then it is guaranteed that $\frac{|X_1\cap X_2|}{k} \geq \lambda$.
\end{example}

This easy case shows how if the data required to perform two studies is sufficiently large relative to the overall dataset, high degrees of data overlap are combinatorially guaranteed.

In the next case, we give an upper bound on $C$ valid in all cases and which is sharp in the case $p = 1$ and $\ell = k$.

\begin{example}
Observe that if $X_i = X_j$ then $|X_i\cap X_j| = |X_i| = k$. Consequently, any $C$ greater than the number of distinct subsets of $D$ of size $k$ guarantees that some $X_i = X_j$ with $i\neq j$. Thus $C = \binom{N}{k} + 1$ guarantees an intersection of size $k$. Hence,
\begin{equation}
\mathrm{pr}\left(\max_{1\leq i<j\leq \binom{N}{k}+1} |X_i \cap X_j| \geq \ell \right) =1
\end{equation}
for all $\ell \leq k$. Moreover, for $p=1$ and $\ell = k$ there is positive probability that all $\binom{N}{k}$ distinct subsets of size $k$ are sampled, so the bound on $C(N,k,k,1)$ is sharp.
\end{example}

\subsection{The Distribution of the Number of Times a Unit is Included in a Study}

How does subsampling with replacement affect the rate of reuse of units? We consider the example of reuse of a control group. Suppose that we have a common control group $C$ and many treatment groups $T_1,\ldots, T_n$. If $n$ distinct studies use $C$ as a control group, what is the effect on the dependence structure of the statistical tasks if each study independently samples some $S_i$ of size $k_i \leq |C|$ without replacement? Let $r_i = \frac{n_i}{|C|}$. Then for each $x\in C$, the probability that $x$ is in the $i^{\text{th}}$ sample is 
\begin{equation}
\mathrm{pr}(x\in S_i) = r_i.
\end{equation}

Assuming independence of the sampling procedures $S_i$, we see that
\[
\mathrm{pr}(x\in S_i \cap S_j)= r_i r_j.
\]

If we know the fraction of control group $r_i$ requested by each study, the distribution of the count of times a given control unit appears in a study is given by $\sum_i \text{Bern}(r_i)$ random variables. If the $r_i$ are sufficiently close to $0$ then this distribution can be approximated by the Poisson distribution with mean $\sum r_i$ by classical results of \cite{cam_approximation_1960} and \cite{jr_poisson_1960}. For instance, suppose that $\max(r_i)\leq \frac{1}{4}$. Then the supremum of the absolute difference $D$ between the cumulative distribution functions of $\text{Poiss}(\sum r_i)$ and $\sum \text{Bern}(r_i)$ satisfies
\begin{equation}
\begin{array}{ccc}
D & \leq & 16 \frac{\sum r_i^2}{\sum r_i}.
\end{array}
\end{equation}
Consequently, for sufficiently small values of $\sum r_i$ the probability that a given unit will be used by more than one study is approximately
\begin{equation}
\mathrm{pr}\left(\sum \text{Bern}(r_i) \geq 2 \right) \approx 1-\frac{1+\sum r_i}{e^{-\sum r_i}}.
\end{equation}

\section{Downstream Effects and Future Directions}

We have identified the problem of dependent inferential procedures arising from data reuse, and adequately addressing it is a community endeavour. In this section we outline starting points for future research in this area, identifying stakeholders and the roles they might play in developing robust protocols to address the needed changes to protect against this kind of dependence.

\subsection{Considerations for Investigators}

The main positive upshot of this work is that in many cases investigators can take concrete steps to mitigate their data reuse. We anticipate that there are many possible procedures that can be proposed which can be implemented by individual investigators to mitigate this form of dependence. Given the proliferous nature of research, the relative tradeoffs between different procedures should be explored. 

For example, consider an investigator evaluating hypothesis $H_0$ comparing the means of two groups. Prior to receiving the data, the investigator can provide an analysis plan including (a) the desired level $\alpha$ for the test of $H_0$, (b) a minimum relevant effect size $\Delta_0$, and (c) a desired power $1-\beta$ of detecting an effect of size $\geq \Delta_0$ to determine the per-arm sample size $(N_0, N_1)$ needed to perform the analysis. Then, provided that the dataset $D$ has sufficient data to be achieve the desired power, the investigator can sample (with replacement) from $D$ to achieve the desired sample size. In case methods to adjust for selection effect such as matching are required, a larger sample may be required; the development of heuristics to determine how much larger the amount of data required would be valuable. The choice of $\alpha$, $\beta$ and $\Delta_0$ should be justified; see for example \cite{lakens_justify_2018}. The evaluation of this tradeoff is especially pertinent for the analysis of rare outcomes and subgroups.

Similarly, for estimation tasks, the investigator can specify a desired coverage probability $1-\alpha$ and precision $\rho$ to determine the required sample size.

This approach has two felicitous knock-on effects. First, it encourages the investigator to engage in research design prior to receiving data, and second it encourages investigators to be explicit in reporting their criteria for inferential success.


\subsection{Considerations for Data Governance and Funding Agencies}

Some large databases are freely and publicly available. Others, more common in clinical outcomes research, may require fees and an analysis plan to be provided prior to the release of the data. For such databases, the database management team can occupy a critical role in reducing inferential dependence. To an extent, some practices currently employed by the centrally-managed datasets provide partial mitigation of these issues by reviewing proposed studies to ensure that facially overlapping studies are not performed, reducing the reuse of data. 

The results of this paper suggest that more robust policies for managing cross-study dependence are possible for these centrally-managed datasets. The database manager has many options for managing the use of their data, including (i) the ability to set an inferential utility function to inform dependence tolerance on a batch of requested studies, and (ii) review the requests of investigators to pressure test their choices of level, power, and minimum relevant effect size or precision, or (iii) setting fixed $\alpha$, $\beta$, and $\Delta_0$ ex ante.

Likewise, funding agencies can also play a role by requiring investigators to justify their data usage strategies and demonstrate how they will minimize unnecessary data reuse while maintaining adequate statistical power. Funding agency level procedures can facilitate across-researcher behaviors through coordination, but also through offering incentives which guide individual-level research practices.

\subsection{Considerations for Journals and Reviewers}

Journals, in their editorial role, are a major checkpoint currently used to verify compliance with best practices against research errors. The current framework to evaluate statistical rigor and reproducibility used at most journals is informed by the traditional framework, where each study has some independence from previous research through the use of novelly generated data sets. Recognizing the new data reuse regime in light of our findings means changing how reviewers and editors perform their roles. The research in this direction will quickly encounter the burden that could be placed on reviewers to be aware of preceding research, and will need to consider ways to mitigate that burden. The development of standardized reporting guidelines relating to data reuse and modeled after successful standards (such as CONSORT) are necessary.

\section{Conclusion}

The growth of scientific knowledge necessarily takes place in a resource-constrained environment, a constraint that in many cases has been significantly alleviated by the economies of scale that allow for the streamlined collection and release of data to investigators in the form of registries and large datasets. However, common practice in the use of these datasets exposes our collective knowledge to cascading risks of inferential error through heavy reuse of data that we term data gluttony.

Actuarial risk theory provides a framework to evaluate these risks from a decision-theoretic perspective to evaluate and identify ways to mitigate these risks. Current standards for statistical reporting are established at a study-by-study, hypothesis-by-hypothesis approach to managing error rates that fail to address the systemic issue of dependence in our statistical procedures. While the per-comparison error rate might be kept to a nominal $\alpha = 0.05$, the shape of the distribution of errors is not controlled.

In this paper we articulated a decision-theoretic approach to managing cascading risks of inferential error due to data reuse. The solution to this turns on the diminishing marginal returns of statistical success in the sample size: one can mitigate one crucial source of inferential error by using only as much data as is needed to achieve a negligibly different level of statistical performance.

This analysis also highlights one of the key benefits of these datasets: the analysis of rare subgroups and outcomes. Depending on the sample size of the group or incidence of the outcome, the tradeoff between ensuring minimally dependent inferences across studies and power may be large. In such cases, the balance between gluttony and temperance may favour the glutton depending on the importance of question at hand.

Additionally, this work reinforces ways in which large datasets may be being used inefficiently: one would expect that the first effects to be validated in a dataset are suspected large effects that are detectable with a sample size orders of magnitude smaller than size of the data registry. Rather, one of the implications of our analysis is that large datasets are unique in their ability to facilitate more data-intensive statistical tasks such as careful evaluations of (i) equivalence or noninferiority of interventions, or (ii) analysis of rare subgroups and rarely occurring outcomes.
\newpage

\bibliographystyle{plain}
\bibliography{refs}

\newpage
\section*{Appendix 1}
\subsection*{Proofs of Results in Distributions of Error}

We state and prove the theorem that the error probabilities in Table \ref{tab:dependent_error_distribution_condprob} are correct.

\begin{theorem}\label{thm:dependent_error_probs}
Let $E_1$ and $E_2$ be events. Then the random variable counting the occurrence of $E_i$ is given by the distribution in Table \ref{tab:dependent_error_distribution_condprob}.
\end{theorem}

\begin{proof}
Observe that the distribution of the count of $E$ is given by
\begin{table*}[h]
\centering
\begin{tabular}{cc}
\textbf{Count} & \textbf{Probability} \\
0 & $\mathrm{pr}(E_1^c \cap E_2^c)$ \\
1 & $\mathrm{pr}((E_1 \cup E_2) \cap (E_1\cap E_2)^c )$ \\
2 & $\mathrm{pr}(E_1 \cap E_2)$ \\
\end{tabular}
\end{table*}
    
By definition, the probability of two errors is:
\begin{align}
\mathrm{pr}(E_{\mathcal{P}} = 2) & = \mathrm{pr}(E_1\cap E_2) \\
& = \mathrm{pr}(E_1)\times \mathrm{pr}(E_2|E_1).
\end{align}

More tricky is the probability of zero errors:
\begin{align}
\mathrm{pr}(E_{\mathcal{P}} = 0) & = \mathrm{pr}(E_1^c \cap E_2^c) \\
& = \mathrm{pr}(E_1^c) \times \mathrm{pr}(E_2^c|E_1^c)\\
& = \mathrm{pr}(E_1^c) \times \left(\frac{\mathrm{pr}(E_2^c) - \mathrm{pr}(E_2^c\cap E_1)}{1-\mathrm{pr}(E_1)} \right) \\ 
& = \mathrm{pr}(E_1^c) \times \left(\frac{\mathrm{pr}(E_2^c) - \mathrm{pr}(E_1)\mathrm{pr}(E_2^c|E_1)}{1-\mathrm{pr}(E_1)} \right) \\
& = (1-\mathrm{pr}(E_1)) \times \left(\frac{1-\mathrm{pr}(E_2) - \mathrm{pr}(E_1)(1-\mathrm{pr}(E_2|E_1))}{1-\mathrm{pr}(E_1)} \right) \\
& = 1-\mathrm{pr}(E_2) - \mathrm{pr}(E_1)(1-\mathrm{pr}(E_2|E_1)).
\end{align}

Finally, since the probability of exactly one error must satisfy
\begin{equation}
\mathrm{pr}(E_{\mathcal{P}} = 1) = 1 - \mathrm{pr}(E_{\mathcal{P}} = 0) - \mathrm{pr}(E_{\mathcal{P}} = 2)
\end{equation}
we have
\begin{align}
\mathrm{pr}(E_{\mathcal{P}} = 1) & = 1 - (\mathrm{pr}(E_1)\times \mathrm{pr}(E_2|E_1)) - (1-\mathrm{pr}(E_2) - \mathrm{pr}(E_1)(1-\mathrm{pr}(E_2|E_1))) \\
& = \mathrm{pr}(E_2) + \mathrm{pr}(E_1)\times(1-2\mathrm{pr}(E_2|E_1))
\end{align}
as desired.
\end{proof}

Next, we calculate the stop-loss premiums of this portfolio in terms of $\mathrm{pr}(E_1)$, $\mathrm{pr}(E_2)$, and $\mathrm{pr}(E_2|E_1)$.

\begin{theorem}\label{thm:dependent_stop_loss}
Let $E_1$ and $E_2$ be events. Then the stop-loss premium of the portfolio $\rho_{E_\mathcal{P}}$ is given by the values in Table \ref{tab:dependent_error_stoploss}.
\end{theorem}

\begin{proof}
We calculate each stop-loss premium using the probability mass function for $E_{\mathcal{P}}$ found in Theorem \ref{thm:dependent_error_probs}.
First,
\begin{align}
\rho_{E_\mathcal{P}}(0) & = \mathbb{E}[(E_\mathcal{P}-0)_{+}] \\
& = \sum_{x > 0} x\mathrm{pr}(E_\mathcal{P} = x) \\
& = 1\times \mathrm{pr}(E_\mathcal{P} = 1) + 2\times \mathrm{pr}(E_\mathcal{P} = 2) \\
& = 1\times(\mathrm{pr}(E_2) + \mathrm{pr}(E_1)(1-2\mathrm{pr}(E_2|E_1))) + 2\times\mathrm{pr}(E_1)\mathrm{pr}(E_2|E_1) \\
& = \mathrm{pr}(E_1) + \mathrm{pr}(E_2).
\end{align}

Next,
\begin{align}
\rho_{E_\mathcal{P}}(1) & = \mathbb{E}[(E_\mathcal{P}-1)_{+}] \\
& = \sum_{x > 1} (x-1)\mathrm{pr}(E_\mathcal{P} = x) \\
& = (2-1)\times\mathrm{pr}(E_\mathcal{P} = 2) \\
& = \mathrm{pr}(E_\mathcal{P} = 2) \\
& = \mathrm{pr}(E_1)\times \mathrm{pr}(E_2|E_1).
\end{align}
Finally,
\begin{align}
\rho_{E_\mathcal{P}}(2) & = \mathbb{E}[(E_\mathcal{P}-2)_{+}] \\
& = \sum_{x > 2} (x-2)\mathrm{pr}(E_\mathcal{P} = x) \\
& = 0
\end{align}
as desired.
\end{proof}

\section*{Appendix 2}
\subsection*{Additional examples}

To illustrate how dependence structure across analyses affects the distribution of the number of errors, consider the following example.

\begin{example}\label{ex:ex_repeated_hypothesis}
Let $H_{01} = H_{02}$ be the same null hypothesis with test statistics $T_{01} = T_{02}$, and let $D$ be a dataset.

Suppose that we test $H_{0i}$ each at level $\alpha$ on the same dataset $D$, with rejection regions $R_i$. By construction, the events $R_i$ are not independent:
\begin{equation}
\mathrm{pr}(R_1|R_2) = 1 \neq \alpha = \mathrm{pr}(R_1)
\end{equation}
but

Therefore, the distribution of the count of Type I errors in the portfolio of testing procedures
\begin{equation}
\mathcal{P} = \{(H_{01},R_1, D),(H_{02},R_2, D) \}
\end{equation}    
is described in the following table:

\begin{table}[h]
\centering
\begin{tabular}{cc}
\textbf{Count} & \textbf{Probability} \\
0 & $1 - \alpha$ \\
1 & 0 \\
2 & $\alpha$ \\
\end{tabular}
\end{table}

Let $E_{\mathcal{P}}$ be the random variable of the count of errors. We then have that the per-comparison error rate is given by  
\begin{align}
PCER(\mathcal{P}) & = \mathbb{E}\left(\frac{\sum{\mathbb{1}_{R_{\alpha}}}}{|\mathcal{P}|}\right)\\
& = \sum_{i}\frac{\mathrm{pr}(R_{\alpha})}{|\mathcal{P}|} \\
& = \frac{\alpha}{2} + \frac{\alpha}{2} \\
&= \alpha
\end{align}
and that
\begin{align}
FWER(\mathcal{P}) & = \mathrm{pr}(E_{\mathcal{P}}\geq 1) \\
& = \mathrm{pr}(E_{\mathcal{P}}= 1) + \mathrm{pr}(E_{\mathcal{P}}= 2) \\
&= 0+ \alpha \\
& = \alpha.
\end{align}
Since the false discovery rate $FDR(\mathcal{P})$ satisfies
\begin{equation}
PCER(\mathcal{P}) \leq FDR(\mathcal{P}) \leq FWER(\mathcal{P})
\end{equation}
we conclude that $FDR(\mathcal{P}) = \alpha$.
\end{example}

Let us contrast this example with testing the same hypothesis on two disjoint datasets.

\begin{example}\label{ex:same_hypothesis_disjoint_data}
Let $H_{01}=H_{02}$ be as in Example \ref{ex:ex_repeated_hypothesis}, but this time tested on disjoint datasets $D_1$ and $D_2$ respectively. Under the global null, the event $R_1$ is independent from the event $R_2$. Therefore the distribution of errors $E_{\mathcal{P}}$ is given by 
\begin{table}[h]
\centering
\begin{tabular}{cc}
\textbf{Count} & \textbf{Probability} \\
0 & $1 - 2\alpha+\alpha^2$ \\
1 & $2\alpha - 2\alpha^2$ \\
2 & $\alpha^2$ \\
\end{tabular}
\end{table}
    
In this case, 
\begin{equation}
PCER(\mathcal{P}) = \alpha
\end{equation}
and 
\begin{equation}
FWER(\mathcal{P}) = 2\alpha - \alpha^2
\end{equation}
but the expected number of errors remains the same.
\end{example}

Comparing the results of Examples \ref{ex:ex_repeated_hypothesis} and \ref{ex:same_hypothesis_disjoint_data}, we observe that decorrelating the tests resulted in a distribution of errors with the same mean but lower probability of 2 errors. Setting $\alpha = 0.05$, we get the following table of error probabilities:

\begin{table}[h]
\centering
\begin{tabular}{ccc}
\textbf{Count} & \textbf{Probability (Identical Tests)} &
\textbf{Probability (Independent Tests)}
\\
0 & $1-\alpha$ & $1 - 2\alpha+\alpha^2$ \\
1 & 0 & $2\alpha - 2\alpha^2$ \\
2 & $\alpha$ & $\alpha^2$ \\
\end{tabular}
\end{table}

In this example, the probability of at least zero errors is higher in the scenario described in Example \ref{ex:ex_repeated_hypothesis} than it is in the scenario of Example \ref{ex:same_hypothesis_disjoint_data}, but at the cost of increased risk of more catastrophic error accumulation

\end{document}